\documentclass[a4paper, 10pt]{article}
\usepackage{mathrsfs}
\usepackage{amsfonts}
\usepackage{latexsym}
\usepackage{graphicx}
\usepackage{amsmath}
\newcommand{\new}{\newcommand*}

\new{\rnew}{\renewcommand*}
\new{\newe}{\newenvironment*}
\new{\newt}{\newtheorem}
\new{\stl}{\setlength}

\new{\bea}{\begin{eqnarray}}
\new{\eea}{\end{eqnarray}}

\new{\be}{\begin{equation}}
\new{\ee}{\end{equation}}

\new{\bean}{\begin{eqnarray*}}
\new{\eean}{\end{eqnarray*}}

\new{\no}{\nonumber}

\new{\bt}{\begin{theorem}}
\new{\et}{\end{theorem}}
\new{\bl}{\begin{lemma}}
\new{\el}{\end{lemma}}

\new{\bc}{\begin{corollary}}
\new{\ec}{\end{corollary}}
\new{\bp}{\begin{proof}\quad}
\new{\ep}{\end{proof}}
\new{\ba}{\begin{array}}
\new{\ea}{\end{array}}

\newt{prop}{Proposition}
\newt{lem}{Lemma}
\newt{conjecture}{Conjecture}[section]
\newt{cor}{Corollary}
\newt{thm}{Theorem}
\newt{assumption}{Assumption}
\newt{ex}{Example}
\newt{remark}{Remark}
\newt{definition}{Definition}

\rnew{\theequation}{\thesection.\arabic{equation}}
\new{\sect}[1]{ \section{#1}
\setcounter{equation}{0} \setcounter{figure}{0} }

\newe{acknowledgment}{{\flushleft \bf Acknowledgments. }}{}
\newe{proof}{{\bf Proof.}}{\qquad\endproof}
\newe{keywords}{\begin{quote}\quad{\bf Keywords.}}{\end{quote}}
\newe{AMS}{\begin{quote}\quad {\bf AMS subject classification.}}{\end{quote}}

\def \endproof {\qquad \vrule height 5pt width 5pt depth 2pt}

\title{The Riemann Mapping Problem\date{}}

\author{{  Zhijian Qiu }
\\{\small   School  of  mathematics}
\\{\small  Southwestern University of Finance and Economics}
\\{\small   Chengdu,  China, E-mail:\ qiu@swufe.edu.cn }
}

\begin{document}
\maketitle
\begin{abstract}
In this article we investigate the century-old continuous extension
problem of the Riemann map.  Let $G$ be a simply connected domain.
We call $\lambda$ in $\partial G$
  a multiple point   if   there are simply
connected subdomains $ U$ and $V$  such that
 $\lambda \in\partial U \cap\partial V$ and $ dist (\partial U\cap G ,
 \partial V\cap G )>0$.  We show that the Riemann map of $G$ has a
continuous extension to $\overline G$ if and only if $\partial G$
has no multiple points.

All of the results in this paper, together with the Riemann mapping
theorem, give a complete and desirable solution to the mapping
problem which was originally raised by Riemann in 1851 and
intensively investigated by many famous mathematicians  throughout
history.
\end{abstract}

\begin{keywords} Riemann map,  simply connected
domain.
\end{keywords}
\begin{AMS} 30H05, 30E10.
\end{AMS}

 \section*{Introduction}
 Complex analysis is not only one of the most outstanding
 accomplishments of classical mathematics, but   a very
 important component in modern analysis.
A connected open subset $G$ in the complex plane {\bf C} is called a
simply connected domain if  $({\bf C}\cup \{\infty\})-G$ is
connected. The most important result in the core part of the theory
is the Riemann mapping theorem, which has been said by some to be
the greatest theorem of the nineteenth century \cite{krantz}[p.83].
The Riemann mapping theorem says that every simply connected domain
having at least two boundary points can be mapped onto the unit disk
$D$ by an injective analytic function $\varphi$. Poincar\'{e} showed
that $\varphi$ is essentially unique. This map is known as the
Riemann map.

The Riemann mapping theorem solves only half of the   problem that
was initially investigated  by Riemann in 1851  \cite{rman}. The
problem of how the map behaves on the boundary  became the focus
afterwards. When $\varphi$ extends continuously to   $\overline G$?
 Schwarz was the first one who
separated out the interior part and the   boundary part of the
Riemann mapping theorem,
 and Poincar\'{e}  meanwhile  did some thing similar.
 Osgood and Taylor   separated the problem  as they
 wrote  in \cite{ot}: "Riemann's problem of mapping
a simply connected plane region whose boundary consists more than a
single point conformally on a circle as normal region may be divided
into two parts:  a) the internal problem; namely, the map of the
interior points, and  b) the boundary problem; namely, the behavior
of the map on the boundary".  In his paper  \cite{cara1912},
Carathe\"{o}dory divided the Riemann mapping problem into two parts:
the interior problem, and the existence of a continuous extension of
that map to the boundaries. The purpose of this paper is to solve
this continuous extension problem, which has  been intensively
investigated by Osgood, Carathe\"{o}dory and others  throughout the
years.

 The  Problem  a) was   eventually  solved
through the work of generations of mathematicians including Riemann,
Schwarz, Poincar\'{e}, Klein, Neumann, Harnack, Osgood,
Carathe\"{o}dory and many others. Koebe,   Rieze and Jej\'{e}r
contributed in the process of refinements of the proofs.

For Problem b), Schwarz and Painlev\'{e} and others  proved that
$\varphi$ is a homeomorphism from $\overline G$ onto $\overline D$
for domains with boundaries made by segments or  piecewisely
analytic curves. Osgood was the first one who gave a correct proof
of the Riemann mapping theorem. That success came after he became
fully aware that the boundary behavior was more complicated than had
been suspected previously.  To tackle the boundary behavior problem,
he divided all simply connected domains into two types: Jordan
domains and non-Jordan domains. He observed that
  the Riemann mapping  would   extend to be a
homeomorphism on the closure if the domain has a piecewisely smooth
boundary and he further conjectured the same result would be true
for all Jordan domains. Meanwhile,   Osgood remarked: it didn't make
sense to ask about   behavior on the boundary for non-Jordan
domains.

A decade after Osgood's conjecture, Carathe\"{o}dory took up the
Riemann mapping problem and published three papers in 1912-1913
\cite{ cara1912,cara1913,cara1913b}.  He solved  Problem a) by
giving the first truly function-theoretic proof for the Riemann
mapping theorem. For Problem b), he focused on the problem of
whether the map could have a continuous extension to the boundary.
The most famous result that came out after that was his proof
 of  the  Osgood conjecture    \cite{cara1913}, and
  it is now   known as the Carathe\"{o}dory theorem: the Riemann map
   $\varphi$ extends to be a homeomorphism from
$\overline G$ onto $\overline D$  if $G$ is a Jordan domain.

 For
non-Jordan domains,   Carathe\"{o}dory might share Osgood's view
somehow and did not overcome the difficulties to offer results of
continuous extension on the boundary. Instead, he inaugurated the
theory  of prime ends in \cite{cara1913b}.  The origins of this
theory are actually to be found in the work of Osgood. Using the
theory of prime ends, Carathe\"{o}dory approached the problem in
such a way that he was able to offer an abstract type of resolution.
As Gray wrote in his article  \cite{gray}[p.84] : \ "This was the
final resolution of a problem originally raised by Riemann".
Before this work, 
the most important result related to the extension  problem was
still the Carathe\"{o}dory theorem.

Carathe\"{o}dory's work was regarded  highly. Even
  75 years after,   A. Shields
wrote in the beginning of \cite{sh}: "the results in these papers
are still of great importance"; and then in page 20: "If there had
been Fields Medals at that time, Carathe\"{o}dory  might have been a
candidate on the basis of this work".

The continuous extension problem remains open before this work, and
so does Problem b). In this article, we tackle the problem with a
conceptual new approach. We did not follow the previous routes,
instead, we introduce a key concept: multiple points. We call
$\lambda\in\partial G$ a multiple point  if  there are simply
connected subdomains $ U$ and $V$ such that
 $\lambda \in\partial U \cap\partial V$ and $ dist (\partial U\cap G,
 \partial V\cap G )>0$.
With multiple points, we now can state  our  extension theorem of
Riemann map: $\varphi$ extends continuously to $\overline G$ if and
only if $\partial G $ has no multiple points. This result reveals
the important, essential and amazing fact: the continuity of
$\varphi$ on $\partial G$ has nothing to do with the smoothness or
roughness of $\partial G$, it depends only on the multiplicity of
points in $\partial G$.

In addition to our extension theorem, Theorem~\ref{t:1}
  also characterizes    continuous Riemann maps.
Theorem~\ref{t:1} tells how the extension of $\varphi$ is done at
each point in $\partial G$.  So our method is localizable and  it
actually tells on which part of $\partial G$ the Riemann map is
continuously extendable.  Modifying slightly, it can be used to
treat conformal maps between non-simply connected domains.
Theorem~\ref{t:3} illustrates what a continuous
 Riemann map should be from a topological point of view. Overall,
  all these results, together with the Riemann mapping theorem,
  completely solve the Riemann mapping problem.

\section{The   Results}
A Jordan curve $\gamma$    is   the image of the interval $[a,b]$
under a continuous function $f$ such that $f$ is injective on
$(a,b)$ and $f(a)=f(b)$. $\gamma $ is called a Jordan arc if it the
image of  an injective continuous function on $[a,b]$.
 A crosscut  of $G$ is a  Jordan arc  whose interior is contained in $G$ but its endpoints are  in $\partial G$.
 A crosscut separates $G$ into two disjoint simply connected domains (one may consult \cite{carl-c} for a proof).

A point  $b$ in $\partial G$  is said to be   accessible   if there
is a Jordan arc $J$ that is contained in $\overline G$ and $J\cap
\partial G =\{b\}$.  Let  $\partial_{a}G$  denote the set of accessible points of $G$
and let $\partial_{n}G=\partial G-\partial_{a}G$.

 The  following  few  lemmas  have been known for more than a
 century, however,  the proofs that  the author was able to find are not easy and are based on other results.
 The  the short and direct proofs below are duo to the author's work of constant
 efforts in years. We include them here  for the seek of self-contained,   the reader's
 convenience  and the  ultimate benefit   of the math society.
   Throughout this article,    $G$ denotes a
bounded simply connected domain,   $\varphi$  and  $\psi$   denote
  the Reimann map of $G$  and the  inverse  of $\varphi$, respectively.
\begin{lem} {\label{l:a} }
Let $f$ be an analytic function on $G$ and let $E\subset\partial G$
 such that $\partial G-E$ is a Jordan arc. If   there is a constant
$c$ such that for each $\lambda\in E$, $f(z)\rightarrow c$ as $z$
approaches to $\lambda$ from the inside of $G$, then $f$ is a
constant function on $G$.
\end{lem}
\begin{proof}
Firstly, if $G$ is a disk, then the conclusion follows easily from
the reflection principle.  For a general $G$, observe that we can
find an arc $J$  which is  a portion of a circle
   such that $J$ is a crosscut  of $G$ and its endpoints are in $E$.  $G-J$ has a component $U$ so that $\partial
U\cap E$   connected. Let $g$ be a conformal map from $D$ onto $U$.
Then the inverse $g^{-1}$ can extend continuously to $J^{\circ} $.\
Let
$L=g^{-1}(J^{\circ})$, then 
the hypothesis implies that  for each $b\in
\partial D-\overline L$, $f\circ g(z)\rightarrow c$ as
$z$ approaches to $b$ from the inside of $\varphi(U)$, and this
infers that $f\circ  g=c$. Therefore $f=c$ on $G$.
\end{proof}

\begin{lem}  \label{l:b}
Let $I$ be an arc on $\partial D$ and let $f$ be a bounded analytic
function on $D $, then there is $\lambda\in I^{\circ}$ such that $f$
has a radical limit at $\lambda$.
\end{lem}
\begin{proof}  Suppose the contrary. For $b\in I$ and let
 $J_{b}$ be the radius ending at $b$, then
  $ \overline {\psi(J_{b})}\cap \partial G $ is connected.
  Let $\Gamma_{b}=  \overline {\psi(J_{b})}\cap \partial G$.
  Note,     for  $n\geq 1$, the number of balls of radius great than $\frac{1}{n}$ and mutually disjoint  in $
  \overline G$ is finite.
 So,  there are   $b_{1}$ and $b_{2}$ in $I$, such that $\Gamma_{b_{1}}\cap \Gamma_{b_{2}}\neq
 \emptyset$. Notice that $J_{b_{1}}\cup J_{b_{2}}$ is a crosscut
 that separates $D$ into two parts, and let $W$ be the part whose
 boundary contains an entire arc which joins $b_{1} $ and $b_{2}$.
 Let $a\in \Gamma_{b_{1}}\cap\Gamma_{b_{2}}$ and set
 $g(z)=\frac{1}{\psi(z)-a}$. $g$ maps $W$ onto a simply connected
 domain $\Omega$, and $\partial \Omega= g(J_{b_{1}})\cap
 g(J_{b_{2}})\cup \{\infty\}$. Let $h$ be the Riemann map that
 maps $\Omega$ onto $D$, and set  $\alpha=h\circ g$,  then   $\partial D= h(J_{b_{1}})\cap
 h(J_{b_{2}})\cup \{h(\infty)\}$.
 So we have that  $ lim_{z\rightarrow w}\alpha(z)=h(\infty)$ for
 $w\in \partial W\cap\partial D$ and it follows   that $\alpha$ must be a
 constant.
This is a contradiction.
\end{proof}

\begin{lem} \label {l:c}
If $\lambda\in\partial_{a}G$ and   $J$ is  a Jordan arc such that
$J\subset\overline G$ and $J\cap \partial G=\{\lambda\}$, then
$\lim_{z\rightarrow \lambda}\varphi(z)$ exists, where the limit is
taken\ along with $J$.
\end{lem}
\begin{proof}
 Suppose  the limit does not exist. Let $I=\overline
{\varphi(J^{\circ})}\cap\partial D$, then $I$ is a subarc. By virtue
of Lemma~\ref{l:b}, we can find a crosscut $L$ such that its
endpoints are on $I^{\circ}$ and $\psi  $ maps $L$ onto an open arc
in $G$. Let $\gamma=\overline{\psi(L )}$.  Then it is not difficult
to see that $\gamma$ is a Jordan curve and $\gamma\cap\partial
G=\{\lambda\}$.
Let $U$  be the domain enclosed by $\gamma$.
 Then $\partial U=\gamma$, and this implies that $\psi=\lambda$ on $\varphi(\gamma)-L$.
So it follows that $\varphi$ is a constant function, a
contradiction.
 \end{proof}  \\

Let $G$, $\lambda$ and $J $ be as above, and let
$\varphi_{J}(\lambda)$ denote the limit in Lemma~\ref{l:c}.

\begin{lem} \label {l:2}
 For $a,b\in \partial_{a}G$, if $J_{1}$ and $J_{2}$ are Jordan arcs such that
   $J_{1}\cap\partial
G=\{a\}$ and $J_{2}\cap\partial G=\{b\}$,
 then $ \varphi_{J_{1}}(a)\neq\varphi_{J_{2}}(b)$.
\end{lem}
\begin{proof}
It is clear that there is a crosscut $J$ of $G$ whose parts near the
endpoints coincide with those of  $J_{1}$ and $J_{2}$, respectively.
If the lemma is not true, then $\overline{\varphi(J)}$ is a Jordan
curve
  and $\overline{\varphi(J)}\cap\partial D =\{\lambda\}$ for some
  $\lambda\in\partial D$.
Let $W$ be the domain enclosed by $\overline{\varphi(J)}$ and let
$U=\varphi^{-1}(W)$. Then $\varphi$ is constantly equal to $\lambda$
on $\partial U\cap \partial G$, so it follows from Lemma~\ref{l:a}
that $\varphi$ is a constant. It is a contradiction.
\end{proof} \\
\begin{definition}
 $\lambda\in \partial G$ is
called a multiple point if for each $i=1,2$, there is a crosscut
$\gamma_{i}$   and a   component  $V_{i}$  of \ $G-\gamma_{i}$, such
that  $V_{1}\cap V_{2}=\emptyset$,  $\lambda \in\partial
V_{1}\cap\partial V_{2}$, and $ dist(\gamma_{1},\gamma_{2})>0 $.
\end{definition}

Let $\partial_{m} G$ to denote the set of multiple points.
 We will show  in Theorem~\ref{t:d} that this definition for a multiple point
is equivalent  to the one in the abstract.

\begin{thm}\label{t:main}
 The Riemann map
$\varphi$ extends continuously  to $\overline G$ if and only if \
\mbox{$\partial_{m} G=\emptyset$}.
\end{thm}
\begin{proof}
Necessity.  Suppose that $\partial_{m}G\neq\emptyset$ and let
$\lambda\in\partial_{m} G$. By definition, for  $i=1,2$, there is a
crosscut $J_{i}$  of $G$ and a domain $V_{i}$  such that $V_{i}$ is
a component of $G-J_{i}$, $V_{1}\cap V_{2}=\emptyset$, $
dist(J_{1},J_{2})>0 $ and $\lambda\in
\partial V_{1}\cap\partial V_{2}$.
 Lemma~\ref{l:2} implies that  each $\varphi(J_{i})$ is a crosscut of $D$ and $dist(\varphi(J_{1}),\varphi(J_{2}))>0
 $. So this infers that  $\overline{ \varphi(V_{1}) }\cap \overline{
 \varphi(V_{2})}
=\emptyset$. It is impossible since $\varphi $ is continuous on
$\overline G$.

Sufficiency.
Suppose that $ \partial_{m}G=\emptyset$. 
We claim that for each $\lambda\in\partial G$,  if
$\{z_{n}\}_{1}^{\infty}\subset G$ and
$\lim_{n\rightarrow\infty}z_{n}=\lambda$, then
$\lim_{n\rightarrow\infty}\varphi(z_{n}) $  exists. Suppose the
contrary, then there is a point $\tau\in \partial G$ and a sequence
$\{z_{n}\}$ such that $z_{n}\rightarrow \tau$ but
$\lim_{n\rightarrow \infty} \varphi(z_{n})$ doesn't exist. Then
$\{\varphi(z_{n})\}$ has at least two cluster points. Let
$\lambda_{i}$, $i=1,2$, be two of those cluster points.
Let $r$ be a small positive number (for example,
 $r<\frac{dist(\lambda_{1},\lambda_{2})}{10}$). Set  $J_{i}= \partial D(\lambda_{i},r)\cap   D$, $i=1,2$
 (where  $D(\lambda_{i},r)=\{z:|z-\lambda_{i}|<r\}$). Then each $J_{i}$ is a crosscut of
 $D$ and it separates $D$ into two
disjoint Jordan domains.
  We use $W_{i}$ to denote the one whose closure
contains $ \lambda_{i} $ and let $G_{i}= \psi(W_{i})$, $i=1,2$, then
$G_{1}\cap\ G_{2}=\emptyset$. Observe that each $G_{i} $ contains a
subsequence of $\{z_{n}\}$, and so $\tau\in\partial
G_{1}\cap\partial G_{2}$. On the other hand, there are two subarcs
$l_{1}$ and $l_{2}$ such that $\partial D- \overline
{D(\lambda_{1},r)}\cup\overline {D(\lambda_{2},r)}=l_{1}\cup l_{2}$.
 Let $\gamma $ be a crosscut of $D$ whose endpoints
are in  both  $l_{1}$ and $l_{2}$   such that  it is distant from
 both of $\overline {D(\lambda_{1},r)}$ and   $\overline {D(\lambda_{2},r)}$. By virtue of Lemma~\ref{l:b}
and Lemma~\ref{l:2}, we can modify $\gamma $ slightly (if necessary)
so that $\varphi^{-1}(\gamma )$ is also a crosscut in $G$ and it is
distant from $\overline G_{i}$, $i=1,2$. Since a crosscut separates
$G$ into two disjoint domains, it follows that $dist(\overline
G_{1},\overline G_{2})>0$. So  by definition  that
$\lambda\in\partial_{m}G$. This contradicts our assumption and hence
the claim proved.

 Now for each $z\in \partial G$, we   define
$\varphi(z)=\lim_{w \rightarrow z}\varphi(w)$, where $w$ is taken
from inside of $G$. Then $\varphi$ is clearly well-defined on
$\overline G$. Let $b$ be an arbitrary point in $\partial G$ and let
$\{z_{n}\} \subset
\partial G$ such that $\lim_{n\rightarrow \infty} z_{n}=b$. For each
$n$, there exists $w_{n}\in G$ such that $|w_{n}-z_{n}|<\frac{1}{n}
\hspace{.03in}\mbox{ and }\hspace{.03in}
|\varphi(z_{n})-\varphi(w_{n})|<\frac{1}{n}.$ Thus
$|\varphi(z)-\varphi(z_{n})|\leq
|\varphi(z)-\varphi(w_{n})|+\frac{1}{n} \rightarrow 0, $ as
$n\rightarrow \infty$,
 So  it follows clearly that  $\varphi$ is continuous on $\overline G$.
\end{proof} \\

An immediate consequence is  the following  famous  theorem of
Caratheo\"{o}dory:

\begin{cor}[Carathe\"{o}dory] \label{c:cara}
If $G$ is a Jordan domain, then  $\varphi$
  extends to be a homeomorphism  from $\overline G$ to $\overline D$.
\end{cor}

\begin{proof}
By definition,   it is clear that a Jordan domain has no multiple
points and  hence $\varphi$ extends continuously to $\overline G$.
Injectivity of $\varphi$ directly follows from Lemma~\ref{l:2} and
therefore $\varphi$ is a homeomorphism.
\end{proof}\\

L. Ahlfors wrote in   \cite{a}[p.232]:  "Unfortunately,
considerations of space do not permit us to include
 a proof of this important theorem (the proof would require a
 considerable amount of preparation)". The theorem he mentions there is
 the Carathe\"{o}dory theorem. Our proof here should
 be   fitted in  standard  text books.

The following result is  from \cite{rmap} and it
 gives a characterization of Jordan domains,
which, together with the idea in Example 1  in \cite{rmap}, are
important steps in the process which leads the author to reach the
  results in this article.
\begin{cor}
   $G$ is a Jordan domain if and only if
$\varphi$   extends  continuously to $\overline G$ and  $\partial G=
\partial_{a} G$.
\end{cor}
\begin{proof}
Necessary. Suppose $G$ is a Jordan domain. Then clearly every point
in $\partial G$ is accessible. Since a Jordan domain has no multiple
boundary points,  the continuity of $\varphi$ to $\overline G$
follows from Theorem~\ref{t:main}.

Sufficiency.  By the hypothesis, the Riemann map $\varphi$ is
continuous on $\overline G$ and  
so $\varphi(b)$ is well-defined for each $b\in \partial G$. Now it
follows  from Lemma~\ref{l:2}   that $\varphi$ is also 1-1, hence
$G$ must be a Jordan domain.
\end{proof} \\

As promised, we now show that the definition of multiple points is
equivalent to the one  given in the abstract which involves no
crosscuts and is purely topological.
\begin{thm} \label{t:d}
$\lambda\in \partial_{m} G$ if and only if there are simply
connected subdomains $ U$ and $V$  such that
 $\lambda \in\partial U \cap\partial V$ and $ dist (\partial U\cap G ,
 \partial V\cap G )>0$.
\end{thm}
\begin{proof}
 Necessity is straightforward. For   sufficiency, suppose that there
are  subdomains  $ U$ $\&$ $V$  such that
 $\lambda \in\partial U \cap\partial V$ and  \ $ dist  \partial U\cap G),
 \partial V\cap G )>0$. We first claim  that $U\cap V=\emptyset$. In fact,
if $Q$ is a component of $U\cap V$, then $\partial Q=(\partial Q\cap
U)\cup (\partial Q\cap \partial U)=(\partial Q\cap \partial V)\cup
(\partial Q\cap \partial U)$, the connectivity implies that
\[0=dist(\partial Q\cap \partial V,  \partial Q\cap \partial U)\geq dist (\partial V\cap
G, \partial U\cap G ).\]   It is a contradiction and the claim is
proved. Let $W=\varphi(U)$, then $W\subset D$ and $\partial
D-\partial W=\cup\ l_{i}$, where each $l_{i}$ is a component of
$\partial D-\partial W$. There is one $l_{j}$ such that $l_{j}\cap
\partial [\varphi(V)]\neq \emptyset$. 
On the other hand, we have that  $\partial W-\partial D=\cup\
p_{i}$, where each $p_{i}$ is a component of $\partial W\cap D$
(note, $\partial W-\partial D =\partial W\cap D$). Evidently, each
$p_{i}$ separates $D$ into disjoint subdomains. Among them there is
one $p_{k}$ such that $\overline p_{k}\cap \overline l_{j}\neq
\emptyset$. 
 $\partial D\cup p_{k} $ is a connected closed set and the complement of
$\partial D \cup p_{k}$ contains both $W$ and $\varphi(V)$.  Let
$W_{1}$ denote the component of the complement of $\partial D \cup
p_{k}$ that contains $W$, then $\partial W_{1}\cap D= p_{k}$. So if
we let $V_{1}=\varphi^{-1}(W_{1})$, then $\partial V_{1}\cap G
=\varphi^{-1}(p_{k})\subset \partial U\cap G$.  Observe that
$\lambda \in\partial V_{1}$ also. Set $\delta=\frac{ dist (\partial
U\cap G,\partial V\cap G)}{10}$. Let
$E=\overline{\varphi^{-1}(p_{k})}$,  then $E$  is compact, so there
are finitely  many disks $D(a_{1}, \delta)$ ,..., $D(a_{n},
\delta)$, where $a_{1}, a_{2},...,a_{n},$ are in $E$, such that
$\cup_{j=1}^{n}\ D(a_{j},\delta) \supset E$. Set $\Omega
=\cup_{j=1}^{n}\ D(a_{j},\delta)$, then this is a connected open
set. Let $\partial_{\infty}\Omega$ be the outer boundary of
$\Omega$, namely,   the boundary of the unbounded component of the
complement of $\Omega$. Notice that $\partial_{\infty}\Omega$ is a
(closed) Jordan curve and if  let $O$ be the component of $G-E$
which contains $V$, then $\partial_{\infty}\Omega\cap O\neq
\emptyset$. Let $J_{1}$ be a component of
$\partial_{\infty}\Omega\cap O$. Since $ \partial
O=(\partial O\cap\partial G)\cup E
$ and $J_{1}\cap E=\emptyset$, it follows that
the endpoints of $J_{1}$ are in $\partial G$ and thus $J_{1}$ is a
crosscut of $G$. It separates $G$ into two parts and one of which
contains $V_{1}$. Similarly, if we repeat the above process for $V$,
then we can get a domain $V_{2}$ and a crosscut $J_{2}$ of $G$ such
that $\lambda\in
\partial V_{2}$, $V\subset V_{2}$, $\partial V_{2}\cap G\subset
\partial V\cap G$, $dist(J_{2},\partial V_{2}\cap G)< \delta $. Now for each $i=1,2$,
if we let $\Omega_{i}$ be the component of $G-J_{i}$ which  contains
$V_{i}$, then $\lambda\in \partial
\Omega_{1}\cap\partial\Omega_{2}$, $dist (J_{1},J_{2})>8\delta$ and
$\Omega_{1}\cap\Omega_{2}=\emptyset$, so it follows by definition
that $\lambda\in\partial_{m}G$.
\end{proof}\\

\noindent{\bf Remark.} We like to point out the hypothesis of
Theorem~\ref{t:d} implies that $U$ and $V$ must be disjoint (as we
show in the proof above). With this theorem, the concept of multiple
points and  Theorem~\ref{t:main} should be understandable to
undergraduate students\\

The following concepts along with Proposition~\ref{p:1} are not
essential in this article, but perhaps it is   worthwhile to present
them here.

 Let $\partial_{p}G=\partial G-\partial_{m}G$ and  we call the elements of
$\partial_{p}G$  prime (boundary) points. This is just a terminology
we introduce and it has no relations with prime ends introduced by
Carathe\"{o}dory. Prime  points are   in contrast to multiple
points.  Intuitively, a multiple point is a boundary point of
multiple number of   disjoint subdomains   while a prime point is
not.

 Let $\lambda\in \partial G_{m}£ý$,  then by definition
there exist  crosscuts $J$ $\&$ $L$  as well as simply connected
subdomains $U$ $\&$ $V$ such that $dist(J,L)>0$,  $ U\cup V
=\emptyset$. $\lambda$ is called a multiple accessible point if
$\lambda\in\partial_{a} U\cap\partial_{a} V $,   a multiple
unaccessible point if $\lambda\in\partial_{n} U\cap\partial_{ n} V
$, and a semi-accessible  point if  $\lambda \in\partial_{a} U \cap\
\partial_{n}V $. We use $\partial_{ma} G$, $\partial_{mn} G$ and    $\partial_{sa}
G$ to denote the set of  multiple accessible points and the set of
multiple unaccessible points  and the set of semi-accessible points,
respectively.
 The sets $\partial_{sa}G$, $\partial_{ma}G$, and $\partial_{mn}G$
  are not necessarily mutually disjoint subsets of $\partial_{m}G$.

\begin{prop}\label{p:1} Let $G$ be a simply connected domain. Then
\[ \partial_{m}G=  \partial_{sa}G\cup \partial_{ma}G\cup
\partial_{mn}G.\]
and consequently
\[ \partial G= \partial_{p}G\cup \partial_{sa}G\cup \partial_{ma}G\cup
\partial_{mn}G.\]
\end{prop}
\begin{proof}
For the first equality, let $a\in\partial_{m}G$. By definition there
exist crosscuts $J$ $\&$ $L$  as well as simply connected subdomains
$G_{1}$ $\&$ $G_{2}$ such that $dist(J,L)>0$,  $G_{1}\cup G_{2}
=\emptyset$.  Now there are two cases: 1). $a\in\partial_{a}G_{1}$:
if $a\in\partial_{a}G_{2}$, then $a\in
\partial_{ma}G$; otherwise, $a\in\partial G_{2}-\partial_{a}G_{2}$,
then we have that $a\in
\partial_{sa}G$.  2).  $a\in \partial G_{1}-\partial_{a}G_{1}$: if
$a\in\partial_{a} G_{2}$, then  we also have $a\in\partial_{sa}G$,
or else if $a\in
\partial_{n}G_{2}$, then this means
that $a\in\partial_{mn}G $. So in summary we have that $a\in
\partial_{sa}G\cup \partial_{ma}G\cup
\partial_{mn}G.$ Consequently, the first equality holds. The second
equality directly follows from the first one and the definition.
\end{proof}\\

Let $\partial_{pa}G=\partial_{a}G-\partial_{sa}G$ and the points in
$\partial_{pa}G$ are called purely accessible points. In the case
that $b\in\partial_{pa}G$, if $\varphi_{J}( b)$ is  same for each
Jordan arc $J$  with the property that $J$ is contained in
$\overline G$ and $J\cap\partial G=\{b\}$,
 then   $\varphi$  has a well-defined boundary value at $b$
 and we denote it by $\varphi (b)$.

Let $E$ be a component of  $  \partial_{n}G$ and let $b\in E$. If
$\lim_{z\rightarrow b} \varphi(z)$   (where $z$ is taken from inside
of $G$) exists and all the limits are the same  for   $b\in E$, then
we say $\varphi$ extends to be a  constant   on  $E$.

The  next theorem   reveals how $\varphi$ behaves on different types
of boundary points.

\begin{thm}\label{t:1}
 $\varphi$ has a continuous extension to $\overline G$ if and only if
the following hold:

\noindent 1) 
every point in $\partial_{a}G$ is purely accessible;

\noindent 2) $\varphi $ has a well-defined boundary value  for each
$b\in\partial_{a} G$;

\noindent 3)  $\varphi $  extends  to be a  constant   on  $E$ for
each  component $E$ of \  $\partial_{n}G$.
\end{thm}

\begin{proof}
Necessity. Suppose that $\varphi$ can extend  continuously to
$\overline G$.  Then clearly $\varphi(z)$ has a  well-defined
boundary value function which is continuous on $\partial G$. By
Lemma~\ref{l:2} we know that  $\varphi$  is injective on
$\partial_{a}G$. Observe that 1) is equivalent to
$\partial_{a}G\cap\partial_{sa} G= \emptyset$. Suppose 1) does not
hold. Then $\emptyset \neq \partial_{a} G\cap\partial_{sa} G \subset
\partial_{m} G$, and this contradicts Theorem~\ref{t:main}. So 1)
must hold. 2) is obvious.

For   the proof of 3), let $E$ be a component
 of $\partial G-\partial_{a} G$ and let $F=\varphi(E)$. It is easy to show that $\overline E^{\circ}=\overline E$ and $E^{\circ}$ is connected.
 Let $b\in F$,  then there is a Jordan arc
 $J_{b}$ in $D$ such that $\overline J_{b} \cap \partial D=\{b\}$, $\varphi^{-1}(\overline {J_{b}})\cap E^{\circ}  \neq
 \emptyset$ and $\varphi^{-1}(\overline {J_{b}})\cap E^{\circ}  $ is
 connected.
  Let $\Gamma_{b}$ denote $\varphi^{-1}(\overline {J_{b}})\cap
  E^{\circ}$ and   let  $B_{b}$ denote an open ball (i.e., an open connected subset of $\partial G$) contained in $\Gamma_{b}$.
Since $\overline E$ is a compact metric space,   for each positive
integer $n$, the number of   $B_{b}$s with radiuses  great than
$\frac{1}{n}$ and mutually disjoint   is finite. So there at most
countable many disjoint $B_{b}$s. Therefore, if we assume that $F$
is not a single point set,  then there are $b_{1}$ and $b_{2}$ in
$F$, such that $B_{b_{1}}\cap B_{b_{2}}\neq \emptyset$, and
consequently, we have that $\Gamma(b_{1}) \cap \Gamma(b_{2})
\neq\emptyset$. For each
 $i=1,2$, there exists a crosscut $l_{i}$ of $ D$ and  a Jordan domain $D_{i}\subset
D$ such that   $dist(l_{1},l_{2})> 0$,     $D_{1}\cap
D_{2}=\emptyset$,  $\partial D_{i}$ consists of $l_{i}$ with an arc
on $\partial D$ and  $b_{i}\in
\partial D_{i}$. By modifying $l_{i}$ slightly if
necessary, we can in addition require that each
$\varphi^{-1}(l_{i})$ is also a crosscut of $G$ and
$dist(\varphi^{-1}(l_{1}),\varphi^{-1}(l_{2}))>0$. Now, let
$V_{i}=\varphi^{-1}(D_{i})$, $i=1,2$, then $V_{1}\cap
V_{2}=\emptyset$. By our constructions of $V_{i}$s,  we see
 that $ \Gamma(b_{1}) \cap \Gamma(b_{2})\subset\partial
 V_{1}\cap \partial V_{2} \subset
\partial _{m} G$. This contradicts the assumption that
$\varphi$ is continuous,
   hence $\varphi$ is constant on $E$.

 Sufficiency. Assume 1), 2) and 3) hold, we show that  $\partial_{m}G=\emptyset$.
 Firstly, it is obvious that 1) is equivalent to
 $\partial_{sa}G=\emptyset$. Secondly,   2) clearly implies $\partial_{ma}G=\emptyset$.
  Lastly, we show that 3) implies $\partial_{mn}G= \emptyset$.
Suppose the contrary, then there is a point $a$ belongs to
$\partial_{mn}G$. By definition there exist crosscuts $J$ $\&$ $L$
as well as simply connected subdomains $G_{1}$ $\&$ $G_{2}$ such
that $dist(J,L)>0$,  $G_{1}\cup G_{2} =\emptyset$ and $a\in \partial
G_{1}\cap\partial G_{2}$. However,   by hypothesis
$\lim_{z\rightarrow a} \varphi(z)$ exists. Let
$\lambda=\lim_{z\rightarrow a} \varphi(z)$ (where $z$ is taken from
inside of $G$)).  Since $\lim_{z\rightarrow a} \varphi(z)$ takes
points from both $G_{1}$ and $G_{2}$, we must have that  $\lambda
\in \overline{ \varphi(G_{1})}\cap \overline{\varphi(G_{2})}$. But,
$dist(\varphi(J),\varphi(L))>0$, so it follows that  $\overline{
\varphi(G_{1})}\cap \overline{\varphi(G_{2})} =\emptyset$.  This is
a contradiction, hence $\partial_{mn}G= \emptyset$. Consequently, we
have $\partial_{m}G=\emptyset$ and therefore $\varphi$ extends
continuously to $\partial G$.
\end{proof}
 \\

\noindent {\bf Remark.} Theorem~\ref{t:1}  really tells how the
extension of $\varphi$ on each part of
 $\partial G$  is done. Modifying our method
  slightly, one can   use it to handle
 conformal maps between non-simply connected domains.

 Osgood once thought that it did not make sense to ask about boundary behavior
 for non-Jordan domains. Using ideas from  Theorem~\ref{t:main} and
 Theorem~\ref{t:1}, we now know that it is possible to define boundary value
 for each point in $\partial G$ except in some extreme cases.  We can
 do so by defining a multiplicity, $m$, for each $a\in \partial G$ as follows: if
$a$ is a prime point, let $m=1$;   else if  $a\in \partial_{m} G$:
if there is an integer $k$ such that  for  any  small disk centered
at $a$,  $G\cap W$ is the union of at most $k$ mutually disjoint
subdomains whose boundary contains $a$, let $n$ be the smallest such
$k$ and then let $m=n$; or else let $m=\infty$. Clearly, in the case
that $a$ has multiplicity $m$, $\varphi(a)$ has $m$ values. It may
not
  make much sense to talk about boundary value in the (extreme) case that
  $m=\infty$. However, one should be  aware that we can actually
  define  a countable many values at a point with $m=\infty$ if we
  accept $\varphi$ as an infinite values function on $\overline G$.
  In that case, $\varphi$ would have values on $\partial G$ for any
  simply connected domain $G$.\\

Let  $G$ be a cornucopia
 domain (see \cite{rs} or \cite{ga})
whose boundary has $\partial D$ as part of it.  It is easy to show
that  $\varphi(z)$ is well-defined  for every   $z\in \overline
G-\partial D$. Notice that $\partial_{m} G=\emptyset$ and $\partial
D =\partial_{n} G$. So by Theorem~\ref{t:main}   $\varphi$ extends
continuously to $\overline G$ and by  Theorem~\ref{t:1} $\varphi$ is
a constant  on $\partial D$, which is   the limit of
$\varphi(z_{n})$ as $\{z_{n}\}$ approaches to a point $\partial D$
from inside of $G$. This shows how the continuous extension of
$\varphi$ is done for a  cornucopia.

 The   cornucopia domains are used in many articles and books, but
 the literatures the author has seen make no comments on the continuity of
 the  Riemann maps. Perhaps  that is because  no  direct proofs can be given easily.
Actually, although the author has known cornucopia domains  three
decades ago, he was aware of this  fact only in recent years. With
Theorem~\ref{t:main}, now one
  can easily and immediately tell if a Riemann map is continuous or not for a domain like a cornucopia.
    \\

The following  example demonstrates that unlike the case of
$\varphi|\partial_{a} G$, $\varphi$ is   not necessarily injective
when $\varphi$ is continuous on $\overline G$.

\begin{ex}
There is a  domain $G$ for which
 $ \partial_{n} G$ has two different components,  $E$ and $F$,   such that
  $\varphi(E)= \varphi(F)$.
\end{ex}

Let $z_{1}=-1$, $z_{2}=1 $ and $z_{3}=i$. For two complex numbers
$z$ and $w$, we use $ \overline{[ z,w]}$ to denote the (closed)
segment from $z$ to $w$. Let $\{u_{n}\}_{n=1}^{\infty}$ be a
sequence of points on the segment $\overline {[z_{1},z_{3}]}$ such
that $Im(u_{1})=0$,  $\{Im (u_{n})\}$ is an increasing sequence and
$u_{n}\rightarrow z_{3}$. For each $n$, let $L_{n}$ be the segment
which is horizontal to the real line and passes through $u_{n}$. Let
$v_{n} =L_{n}\cap \{z:Re(z)=-1\}$.
 For each $n\geq 1$, 
Let $l_{n}$ be the  the path which is the union of $\overline
{[u_{n}-v_{n+1}]}$ with $\overline {[v_{n+1}-u_{n+1}]}$. Now set
\[J=(\cup_{n=1}^{\infty}\ l_{n})\cup[-1, 0]\cup\overline{[0, i]}\cup
\overline{[ i,-1+i]}.\] Then $J$ is a connected closed subset whose
complement consists of exactly a bounded connected
 simply  domain and a unbounded
domain. Let $V$   denote the bounded one. Set $W$ be the reflecting
image of $V$ with respect to the imaginary axis and let $G$ be the
interior of $\overline{V\cup W}$. 
  Let $I =\overline {[-1+i,1+i]} $. Then $I\subset\partial
G$.
 Note,  $I-{i}$ is the union of intervals  $E$ and $F$, where $E=[-1+i,i)$
 and $F=(i,1+i]$.
   From the construction,  we see that $i\in \partial_{a} G$ and
  $(E\cup F)=\partial_{n}G$. Thus, both $E$ and $F$
 are the components of $\partial G_{n}$. Clearly we have that
 $\partial_{m}G=\emptyset$ and so it follows by Theorem~\ref{t:main}
that $\varphi$ extends continuously  to $\overline G$  (one can
directly prove this using the method similar to the one used in
Example 1 in \cite{rmap}).
Therefore,  we have $\varphi(E)=\varphi(F)$.\\

 Nevertheless, we    have the following result, which
further shows how   a continuous Riemann map of $G$ behaves on the
boundary of $G$.

\begin{thm}\label{t:3}
If $\varphi$ is continuous on $\overline G$, then each
 $\varphi^{-1} (z)$ is   a connected closed subset.
 Moreover, let $X=\{\varphi^{-1} (z)\in \overline D\}$  and define
  $\hat\varphi$   by $\hat\varphi(\overline x)\rightarrow\varphi(x)$
for each $\overline x\in X$,
 then   given $X$ with quotient topology,
 $\hat\varphi$ is a homeomorphism  from $X$ onto $\overline
D$.
\end{thm}

\begin{proof}
Assume that  $\varphi$ extends to be continuous on $\overline G$.
Let $b$ be a point in $\overline D$ and assume that $\varphi^{-1}
(b)$ is not connected.
 Then  there exists a component of $\varphi^{-1} (b)$, say $E $,
such that $\partial  G \cap E \neq \emptyset$. Let $a\in
 \partial  G \cap E $ and let $\{v_{n}\}\subset G$ such that
$v_{n}\rightarrow a$. Set $u_{n}=\varphi(v_{n})$
 then   $u_{n}\rightarrow \varphi(a)=b$. For seek of
convenience, we may assume that $b=i$ and in addition require
$Re(u_{n})\leq 0$ for each $n\geq 1$.
 So we can find a Jordan arc $J$ in $D$ which
lies in the left half plane and  contains a subsequence of
$\{u_{n}\}$. Now we consider the right (open) half unit disk.
\[\mbox{Case 1}:\hspace{.2in}\overline{\varphi^{-1}(\{w: Re(w)>0\hspace{.05in} \&\hspace{.05in} w\in
  D\}) }\cap(\varphi^{-1} (i) - E )\neq \emptyset.\] Then, similar to the
 process above, we can find a  Jordan (open) arc $L$ contained in the right
disk such that $\overline{\varphi^{-1}(L)}\cap(\varphi^{-1} (i) - E
)\neq \emptyset$. So    we also are able to find a Jordan curve
$\gamma\subset D$ such that $J\cup L\subset \gamma$ and $\gamma\cap
\partial D=\{i\}$.
 \[\mbox{Case 2}:\hspace{.2in}\overline{\varphi^{-1}(\{w: Re(w )>0\hspace{.03in}
\&\hspace{.03in} w\in  D\})} \cap(\varphi^{-1} (i) - E )=
\emptyset.\]
 In this case, we must have
$\overline{\varphi^{-1}(\{w: Re(w)>0\hspace{.03in} \&\hspace{.03in}
w\in
 D\})} \cap E  \neq\emptyset.$ Meanwhile, we  also have that
$\varphi^{-1}(\{w: Re(w)\leq 0\hspace{.03in} \&\hspace{.03in} w\in D
\})\cap(\varphi^{-1} (i) - E )\neq \emptyset$. So there is a
component $F$ such  that $\varphi^{-1}(\{w: Re(w)\leq
0\hspace{.03in} \&\hspace{.03in} w\in D \})\cap F\neq \emptyset$.
 Now if we start the next process as in the
beginning, by treating $F$ as $E$ and $E $ as $(\varphi^{-1} (i)-
F)$, then we will still be able to get the Jordan curve $\gamma$ as
  demonstrated  above.

Now, let $V$ denote the Jordan domain enclosed by $\gamma$ and let
$W=\varphi^{-1}(V)$. Then
 $\partial V-\{i\}=\gamma-\{i\}\subset D$ and so $\varphi^{-1}(\partial
V-\{i\})\subset G$. But this would force that $\varphi(\partial
W\cap\partial G)=\{i\}$. So
 $\lim_{w\rightarrow\lambda}\varphi(w)=i$   for
each $\lambda\in\partial W\cap\partial G$, where the limit is taken
from inside of $W$, and it follows from Lemma~\ref{l:a} that
$\varphi$ is a constant. This is  impossible and hence
$\varphi^{-1}(z)$ is either a single point or is connected for each
$z\in \overline D.$

Finally, since $\varphi$ now is a continuous surjective map from a
compact   space to a compact space,
 it maps closed subsets to   closed subsets and
therefore it is quotient map. So it follows  that $\hat\varphi$ is a
homeomorphism.
\end{proof}
\\

This theorem says  that $\psi$ is the restriction to $D$ of a
homeomorphism from $\overline D$ to some compact space (the quotient
space) with the topology whose restriction to $G$ is the uniform
topology. Conversely, if $\psi$ extends to be a homeomorphism on
$\overline D$, then it is not difficult to show that $G$ satisfies
the hypothesis of Theorem~\ref{t:1} and hence $\varphi$ extends
continuously to $\overline G$.
\\

    \noindent {\bf Acknowledgement.}
    The author  wishes to express his appreciations to  J. Carmona, K.
Fedorovskiy and J. Xia for helpful comments to some earlier
versions.


\end{document}